\UseRawInputEncoding
\documentclass[11pt]{article}
\usepackage{amsmath}
\usepackage{amssymb, amscd, amsthm,amsfonts}
\usepackage[all]{xy}
\usepackage[dvips]{graphicx}
\usepackage{verbatim}
\usepackage[perpage,symbol*]{footmisc}
\newtheorem{theorem}{Theorem}
\newtheorem{lemma}{Lemma}

\newtheorem{corollary}{Corollary}
\newtheorem{conjecture}{Conjecture}

\begin{document}
	
	\pagestyle{myheadings}

	\title{\bf On the girth cycles of the bipartite graph $D(k,q)$ \footnote{This work was supported by the National Natural Science Foundation of China (No. 61977056).}}
	\author{Ming Xu, Xiaoyan Cheng and Yuansheng Tang\footnote{Corresponding author.}\\
{\it\small School of Mathematical Sciences, Yangzhou University, Jiangsu, China\footnote{Email addresses: mx120170247@yzu.edu.cn(M. Xu), xycheng@yzu.edu.cn(X. Cheng), ystang@yzu.edu.cn(Y. Tang)}}
           }
	\date{}
	\maketitle

    {\noindent\small{\bf Abstract:}
     For integer $k\geq2$ and prime power $q$, the algebraic bipartite graph $D(k,q)$ proposed by Lazebnik and Ustimenko (1995) is meaningful not only in extremal graph theory but also in coding theory and cryptography. This graph is $q$-regular, edge-transitive and of girth at least $k+4$.
     For its exact girth $g=g(D(k,q))$,
     F\"{u}redi et al. (1995) conjectured $g=k+5$ for odd $k$ and $q\geq4$.
     This conjecture was shown to be valid in 2016 when $(k+5)/2$ is the product of an arbitrary factor of $q-1$ and an arbitrary power of the characteristic of $\mathbb{F}_q$.
     In this paper, we determine all the girth cycles of $D(k,q)$ for $3\leq k\leq 5$, $q>3$, and those for $3\leq k\leq8$, $q=3$.
        }

    \vspace{1ex}

\section{Introduction}
   The graphs considered in this paper are undirected, without loops and multiple edges.
   The vertex set and edge set of a graph $G$ are denoted by  $V(G)$ and $E(G)$, respectively.
   For distinct vertices $v,v'\in V(G)$ we write $v\sim_G v'$, or $v\sim v'$ for brevity, iff they are adjacent in $G$, that is, $\{v,v'\}\in E(G)$ is an edge of $G$.
   An automorphism of $G$ means a bijection $\phi$ from $V(G)$ to itself such that $\phi(v)\sim\phi(v')$ iff $v\sim v'$.
   If for any two edges $\{v_1,v_1'\}$, $\{v_2,v_2'\}$ of $G$ there is an automorphism $\phi$ of $G$ such that $\{\phi(v_1),\phi(v_1')\}=\{v_2,v_2'\}$, then $G$ is said to be {\it edge-transitive}.
   A {\it backtrackless} (or {\it non-recurrent}) walk of length $k$ is a sequence $v_1,v_2,\ldots,v_k$ in $V(G)$ such that $v_i\sim v_{i+1}$ for $i=1,2,\ldots,k-1$, and $v_j\neq v_{j+2}$ for $j=1,2,\ldots,k-2$.
   Furthermore, a backtrackless walk $v_1,v_2,\ldots,v_k$ is called a {\it backtrackless circuit} iff its length is greater than 2 and $v_3,v_4,\ldots,v_k,v_1,v_2$ is still a backtrackless walk.
   We note that a $k$-cycle is indeed a backtrackless circuit $v_1,v_2,\ldots,v_k$ consisting of distinct vertices. Clearly, any backtrackless circuit of length less than $2g(G)$ must be a cycle, where $g(G)$ is the girth of $G$, i.e. the length of the shortest cycles in $G$.

   In literature, graphs with large girth and a high degree of symmetry have been applied to variant problems in extremal graph theory, finite geometry, coding theory, cryptography, communication networks and quantum computations (c.f. \cite{Wenger91}--\cite{Dehghan20}). In particular, bipartite garphs are often used to represent systems in science and engineering, where the two sides of the bipartition represent variables and local constraints involving their adjacent variables, respectively. The cycle distribution of the corresponding graph plays an important role in those graphical representations. As an example, the Tanner graph of a low-density parity-check (LDPC) code can be represented by a bipartite graph, where the variable nodes represent code symbols and the constraint nodes represent the parity-check equations. The performance of iterative decoding algorithms depends highly on the cycle distribution and the girth of the Tanner graph. Due to this close relationship, it is of great significance to determine the girth cycles of bipartite graphs with large girth.

   In this paper, we concetrate on the girth cycles of the bipartite graph $D(k,q)$ which was originally proposed by Lazebnik and Ustimenko in \cite{Lazebnik95}, where $k$ is an integer not less than 2 and $q$ is a prime power. The graph $D(k,q)$ has been investigated quite well (e.g. \cite{Lazebnik95}--\cite{XWY16}), in particular it has been proved to be edge-transitive and of girth at least $k+4$ for $k\geq3$.
   On the application of $D(k,q)$ to coding theory and cryptography, we note that there were quite a few works devoted to constructing LDPC codes based on $D(k,q)$ (e.g. \cite{Kim04,Yan11,Polak13}).
   For the exact girth of $D(k,q)$, the following conjecture was proposed in \cite{Fredi95}:
   \begin{conjecture}\label{conmain}
     $D(k,q)$ has girth $k+5$ for all odd $k$ and all $q\geq4$.
   \end{conjecture}
   \noindent
   This conjecture was shown to be valid in \cite{Fredi95} for the case that $(k+5)/2$ divides $q-1$ and in \cite{XWY14} for the case that $(k+5)/2$ is a power of the characteristic of $\mathbb{F}_q$, respectively. The latest result on this conjecture was given in \cite{XWY16}, therein Conjecture \ref{conmain} was shown to be valid for the case that $(k+5)/2$ is the product of a factor of $q-1$ and a power of the characteristic of $\mathbb{F}_q$. To our knowledge, almost all of the known researches on this conjecture are constructive, namely the main conclusions were shown by construction of some girth cycles of the corresponding graphs.

   We will determine all the girth cycles of $D(k,q)$
   for a few small $k$'s in this paper. Instead of working with the original graph $D(k,q)$, we consider the bipartite graph $\Lambda_{k,q}$ proposed in \cite{XWY14}, which is isomorphic to $D(k,q)$ and defined as follows. Let $L_k$ be the set of $(k+1)$-dimensional vectors $(l_0,l_1,l_2,\ldots,l_k)$ over $\mathbb{F}_q$ with $l_1=l_2$. Let $R_k$ be the set of $(k+1)$-dimensional vectors $(r_0,r_1,r_2,\ldots,r_k)$ over $\mathbb{F}_q$ with $r_1=0$. The vectors in $L_k$ and $R_k$ are denoted by $[l]$ and $\langle r \rangle$, respectively. Then $\Lambda_{k,q}$ is the bipartite graph with vertex set $V(\Lambda_{k,q})=L_k\cup R_k$ and edge set $E(\Lambda_{k,q})\subset L_k\times R_k$ such that $[l]=(l_0,l_1,\ldots,l_k)\in L_k$ and $\langle r \rangle=(r_0,r_1,\ldots,r_k)\in R_k$ are adjacent in $\Lambda_{k,q}$ if and only if, for $2\leq i \leq k$,
   \begin{align}\label{adj_con}
      l_i+r_i=\begin{cases}
         r_0l_{i-2} \text{ ~~if } i\equiv2,3 \mod 4,\\
         l_0r_{i-2} \text{ ~~if } i\equiv0,1 \mod 4.
      \end{cases}
   \end{align}
Since $\Lambda_{k,q}$ and $D(k,q)$ are isomorphic graphs \cite{XWY14}, $\Lambda_{k,q}$ is also edge-transitive and of girth at least $k+4$ for $k\geq 3$.

   This paper is arranged as follows.
   In Section 2 we show a closed-form expression for the backtrackless walks of $\Lambda_{k,q}$ which are leading by the all-zero vectors $[l]=(0,0,\ldots)$ and $\langle r \rangle=(0,0,\ldots)$.
   By using this expression, all the girth cycles in $\Lambda_{3,q}$, and those in $\Lambda_{4,q}$ for $q>3$ are determined in Section 3.
   In Section 4, we present a necessary and sufficient condition for some backtrackless walks of length 10 to be circuits of $\Lambda_{5,q}$, and thus all the girth cycles of $\Lambda_{5,q}$ are determined.
   For $4\leq k\leq 8$, all the girth cycles of $\Lambda_{k,3}$ are determined in Section 5.
   Some concluding remarks are given in Section~6.

\section{Backtrackless Walks in $\Lambda_{k,q}$}
   Since $\Lambda_{k,q}$ is edge-transitive, without loss of generality one can deal with only the cycles which contains the edge $[l]=(0,0,\ldots,0)\sim\langle r\rangle=(0,0,\ldots,0)$.
   Let $\Gamma=[l^{(1)}]\langle r^{(1)}\rangle[l^{(2)}]\langle r^{(2)}\rangle\cdots$ be a given backtrackless walk of the bipartite graph $\Lambda_{k,q}$ leading by the all-zero vectors $[l^{(1)}]=(0,0,\ldots,0)$ and $\langle r^{(1)}\rangle=(0,0,\ldots,0)$. Let $x_i$ and $y_i$ denote the first entries (or colors) of $[l^{(i)}]$ and $\langle r^{(i)}\rangle$, respectively.
   For $i\geq1$, let
   \begin{align}\label{relationxy}
      u_i=x_{i+1}-x_{i}, ~~~~~~~v_i=y_{i+1}-y_{i}.
   \end{align}
   Clearly, we have $u_i\neq0$ and $v_i\neq0$.
   For $i\geq1$ and $j\geq0$, let $l_j^{(i)}$ and $r_j^{(i)}$ denote the $(j+1)$-th entries of $[l^{(i)}]$ and $\langle r^{(i)}\rangle$ respectively.

   At first, we introduce the notation $\rho_s(\omega_1,\ldots,\omega_n)$ proposed in \cite{XWY14} which is useful for expressing the vertices in the walk $\Gamma$ with their colors.
   For $\omega_1,\ldots,\omega_n\in\mathbb{F}_q^{\ast}$, let
   $$ \rho_0(\omega_1,\ldots,\omega_n)=\omega_1\cdots\omega_n$$
   and, for $1 \leq s\leq \lfloor\frac{n}{2}\rfloor$, let
   $$ \rho_s(\omega_1,\ldots,\omega_n)=\sum_{1\leq i_1<\cdots<i_s\leq n-s}\frac{\prod_{j=1}^n\omega_j}{\prod_{j=1}^s\omega_{i_j+j-1}\omega_{i_j+j}},$$
   where each term in the summation is a product of the remaining elements in the sequence $\omega_1,\ldots,\omega_n$ after deleting $s$ disjoint pairs $\{\omega_i,\omega_{i+1}\}$ of consecutive elements.
   If $n<2s$ or $s<0$, $\rho_s(\omega_1,\ldots,\omega_n)$ is defined as 0. For the null sequence $\eta$, $\rho_s(\eta)$ is defined as
   \begin{align*}
      \rho_s(\eta)=\begin{cases}
         1 \text{ ~~if } s=0,\\
         0 \text{ ~~if } s\neq 0.
      \end{cases}
   \end{align*}
   From the definition of $\rho_s(\omega_1,\ldots,\omega_n)$, one can show easily
   \begin{align}\label{rhoproperty}
      \rho_s(\omega_1,\ldots,\omega_n)=\rho_{s-1}(\omega_1,\ldots,\omega_{n-2})+\omega_n\rho_s(\omega_1,\ldots,\omega_{n-1}),
   \end{align}
   and, for $0\leq j\leq n$,
   \begin{align}
      \rho_{n-j}(\omega_1,\ldots,\omega_{2n})=\sum_{0\leq s_1<t_1\leq s_2<t_2\leq \cdots \leq s_j<t_j\leq n}\prod_{k=1}^j\omega_{2s_k+1}\omega_{2t_k},\label{0b0}
   \end{align}
   \begin{align}
      \rho_{n-j}(\omega_1,\ldots,\omega_{2n+1})=\sum_{0\leq s_0<t_1\leq s_1<t_2\leq \cdots \leq s_{j-1}<t_j\leq s_j\leq n}\omega_{2s_0+1}\prod_{k=1}^j\omega_{2t_k}\omega_{2s_k+1}.\label{0b1}
   \end{align}

By using of the notation $\rho_s(\omega_1,\ldots,\omega_n)$, a closed-form expression for the backtrackless walks leading by the all-zero vector $[l^{(1)}]=(0,0,\ldots,0)$ was given in \cite{XWY14}.
Since the second vertex in the walk $\Gamma$ is also the all-zero vector $\langle r^{(1)}\rangle=(0,0,\ldots,0)$, we improve the closed-form expression further in the following theorem.
   \begin{theorem}\label{path}
       For any $i\geq 1$ and $j\geq  0$, we have
       \begin{align}
          l_{4j}^{(i+1)}&=\rho_{i-j-1}(u_1,v_1,\ldots,u_{i-1},v_{i-1},u_i),\label{4j}\\
          l_{4j+1}^{(i+1)}&=\rho_{i-j-2}(v_1,u_2,\ldots,v_{i-1},u_i),\label{4j+1}\\
          l_{4j+2}^{(i+1)}&=y_{i+1}l_{4j}^{(i+1)}-\rho_{i-j-1}(u_1,v_1,\ldots,u_i,v_i),\label{4j+2}\\
          l_{4j+3}^{(i+1)}&=y_{i+1}l_{4j+1}^{(i+1)}-\rho_{i-j-2}(v_1,u_2\ldots,v_{i-1},u_i,v_i).\label{4j+3}
       \end{align}
   \end{theorem}

   \begin{proof}
      Since $[l^{(2)}]=(x_2,0,\ldots)=(u_1,0,\ldots)$, one can check easily that (\ref{4j}--\ref{4j+3}) are valid when $i=1$. Assume (\ref{4j}--\ref{4j+3}) are valid when $i=t$ for some $t\geq 1$. For $j\geq0$, from $[l^{(t+1)}]\sim\langle r^{(t+1)}\rangle$ we see
      \begin{align*}
         r_{4j+2}^{(t+1)}&=r_{0}^{(t+1)}l_{4j}^{(t+1)}-l_{4j+2}^{(t+1)}=\rho_{t-j-1}(u_1,v_1,\ldots,u_t,v_t),\\ r_{4j+3}^{(t+1)}&=r_{0}^{(t+1)}l_{4j+1}^{(t+1)}-l_{4j+3}^{(t+1)}=\rho_{t-j-2}(v_1,u_2,\ldots,v_{t-1},u_t,v_t).
      \end{align*}
      Furthermore, for $j\geq 1$ from $[l^{(t+1)}]\sim\langle r^{(t+1)}\rangle\sim[l^{(t+2)}]$ and (\ref{rhoproperty}) we have
      \begin{align}
      	 l_{4j}^{(t+2)}&=l_{0}^{(t+2)}r_{4j-2}^{(t+1)}-r_{4j}^{(t+1)}\nonumber\\
      	              &=(l_{0}^{(t+2)}-l_{0}^{(t+1)})r_{4j-2}^{(t+1)}+l_{4j}^{(t+1)}\nonumber\\
      	              &=u_{t+1}\rho_{t-j}(u_1,v_1,\ldots,u_t,v_t)+\rho_{t-j-1}(u_1,v_1,\ldots,u_{t-1},v_{t-1},u_t)\nonumber\\
      	              &=\rho_{t-j}(u_1,v_1,\ldots,u_t,v_t,u_{t+1}),\label{induction4jt+2}\\
      	 l_{4j+1}^{(t+2)}&=l_{0}^{(t+2)}r_{4j-1}^{(t+1)}-r_{4j+1}^{(t+1)}\nonumber\\
      	 &=(l_{0}^{(t+2)}-l_{0}^{(t+1)})r_{4j-1}^{(t+1)}+l_{4j+1}^{(t+1)}\nonumber\\
      	 &=u_{t+1}\rho_{t-j-1}(v_1,u_2,\ldots,v_{t-1},u_t,v_t)+\rho_{t-j-2}(v_1,u_2,\ldots,v_{t-1},u_t)\nonumber\\
      	 &=\rho_{t-j-1}(v_1,u_2,\ldots,v_t,u_{t+1}).\label{induction4j+1t+2}
      \end{align}
      From $l_0^{(t+2)}=x_{t+2}=u_1+\cdots+u_{t+1}=\rho_t(u_1,v_1,\ldots,u_t,v_t,u_{t+1})$ and
      \begin{align*}
      	l_1^{(t+2)}=l_2^{(t+2)}&=r_{0}^{(t+1)}l_0^{(t+2)}-r_2^{(t+1)}\\
      	&=(l_{0}^{(t+2)}-l_{0}^{(t+1)})r_0^{(t+1)}+l_2^{(t+1)}\\
      	&=u_{t+1}y_{t+1}+\rho_{t-2}(v_1,u_2,\ldots,v_{t-1},u_t)\\
      	&=u_{t+1}\sum_{1\leq s\leq t}v_s+\sum_{1\leq s_1<s_2\leq t}v_{s_1}u_{s_2}=\rho_{t-1}(v_1,u_2,\ldots,v_t,u_{t+1})
      \end{align*}
      we see (\ref{induction4jt+2}) and (\ref{induction4j+1t+2}) are also valid for $j=0$. Hence, for $j\geq 0$ we have
      \begin{align*}
         l_{4j+2}^{(t+2)}&=r_{0}^{(t+1)}l_{4j}^{(t+2)}-r_{4j+2}^{(t+1)}\\
         &=y_{t+1}\rho_{t-j}(u_1,v_1,\ldots,u_t,v_t,u_{t+1})-\rho_{t-j-1}(u_1,v_1,\ldots,u_{t},v_{t})\\
         &=y_{t+2}l_{4j}^{(t+2)}-\rho_{t-j}(u_1,v_1,\ldots,u_{t+1},v_{t+1}),\\
         l_{4j+3}^{(t+2)}&=r_{0}^{(t+1)}l_{4j+1}^{(t+2)}-r_{4j+3}^{(t+1)}\\
         &=y_{t+1}\rho_{t-j-1}(v_1,u_2,\ldots,v_t,u_{t+1})-\rho_{t-j-2}(v_1,u_2,\ldots,v_{t-1},u_{t},v_{t})\\
         &=y_{t+2}l_{4j+1}^{(t+2)}-\rho_{t-j-1}(v_1,u_2,\ldots,u_{t+1},v_{t+1}).
      \end{align*}
      Therefore, according to induction, we see (\ref{4j}--\ref{4j+3}) are valid for any $i\geq 1$.
   \end{proof}

   The walk $\Gamma$ will be called of type $(u_1,v_1,u_2,v_2,\ldots)$. If the first $2i$ vertices in the walk $\Gamma$ form a circuit in $\Lambda_{k,q}$ of length $2i$, we also say it is a backtrackless circuit of type $(u_1,v_1,\ldots, u_i,v_i)$. In the next sections, we will deduce some conditions for the first vertices in $\Gamma$ forming a circuit in $\Lambda_{k,q}$ for some small $k$'s, and then determine all the girth cycles in these graphs.

\section{Girth Cycles of $\Lambda_{3,q}$ and $\Lambda_{4,q}$}

   Since the graph $\Lambda_{k,q}$ has girth at least $k+4$,
   we see that any backtrackless circuit of length 8 in $\Lambda_{k,q}$ is a cycle, namely consisting of distinct vertices.
   According to Theorem~\ref{path}, the first vertices of $\Gamma$ form a cycle of type $\epsilon=(u_1,v_1,\ldots,u_4,v_4)$ in $\Lambda_{3,q}$ if and only if $v_1,v_2,v_3,u_2,u_3,u_4\in\mathbb{F}_q^{\ast}$ satisfy $v_4=-v_1-v_2-v_3\neq 0$, $u_1=-u_2-u_3-u_4\neq 0$ and
   \begin{align}\label{lamda3q}
     \begin{cases}
       v_1u_2+(v_1+v_2)u_3+(v_1+v_2+v_3)u_4=0,\\
       v_1^2u_2+(v_1+v_2)^2u_3+(v_1+v_2+v_3)^2u_4=0.
     \end{cases}
   \end{align}
   Hence, the cycles of length 8 in $\Lambda_{3,q}$ can be determined simply by solving the linear system (\ref{lamda3q}).

   If $v_1+v_2=0$, then one should set $v_3=v_1$, $u_4=-u_2$, $v_4=-v_1$, and $u_1=-u_3$. Hence, for any $r,s,t\in\mathbb{F}_q^{\ast}$, let
   \begin{align}\label{3v_12eq0}
   \epsilon=(r,s,t,-s,-r,s,-t,-s),
   \end{align}
   then we get the following backtrackless walk of length 8 in $\Lambda_{4,q}$
   \begin{gather*}
      [0,0,0,0,0]\sim\langle0,0,0,0,0\rangle\sim[r,0,0,0,0]\sim\langle s,0,rs,0,r^2s\rangle\sim\\
      [r+t,st,st,s^2t,rst]\sim\langle0,0,-st,-s^2t,-st(2r+t)\rangle\sim\\
      [t,st,st,s^2t,2rst]\sim\langle s,0,0,0,-2rst\rangle,
   \end{gather*}
   which gives a cycle of length 8 in $\Lambda_{3,q}$, and a cycle of length 8 in $\Lambda_{4,q}$ if and only if the characteristic of $\mathbb{F}_q$ is 2.

   If $v_1+v_2\neq0$, then we should have $v_2+v_3\neq0$, $v_1+v_2+v_3\neq0$ and thus there is a $t\in\mathbb{F}_q^{\ast}$ such that
   \begin{align*}
      u_2&=t\begin{vmatrix}
          v_1+v_2 & v_1+v_2+v_3\\
          (v_1+v_2)^2 & (v_1+v_2+v_3)^2
          \end{vmatrix},\\
      u_3&=-t\begin{vmatrix}
          v_1 & v_1+v_2+v_3\\
          v_1^2 & (v_1+v_2+v_3)^2
          \end{vmatrix},\\
      u_4&=t\begin{vmatrix}
          v_1 & v_1+v_2\\
          v_1^2 & (v_1+v_2)^2
          \end{vmatrix}.
   \end{align*}
   Hence, for distinct $a,b,c\in\mathbb{F}_q^{\ast}$, let
   \begin{gather}
   v_1=a,v_2=b-a,v_3=c-b,v_4=-c,\label{03x0}\\
   u_2=tbc(c-b),u_3=tac(a-c),u_4=tab(b-a),\label{03x1}\\
   u_1=-(u_2+u_3+u_4)=t(c-b)(a-c)(b-a),\label{03x2}
   \end{gather}
   where $t\in\mathbb{F}_q^{\ast}$, then we get a backtrackless walk of length 8 in $\Lambda_{4,q}$ as the following
   \begin{footnotesize}
   \begin{gather*}
   [0,0,0,0,0]\sim\langle0,0,0,0,0\rangle\sim[t(c-b)(a-c)(b-a),0,0,0,0]\sim\\
   \langle a,0,ta(c-b)(a-c)(b-a),0,t^2a(c-b)^2(a-c)^2(b-a)^2\rangle\sim\\
   [ta(c-b)(b+c-a),tabc(c-b),tabc(c-b),ta^2bc(c-b),t^2abc(c-b)^2(a-c)(b-a)]\sim\\
   \langle b,0,tab(c-b)(b-a),tabc(c-b)(b-a),t^2ab(c-b)^2(b-a)(ab-a^2+c^2)\rangle\sim\\
   [tab(a-b),tabc(a-b),tabc(a-b),tabc^2(a-b),t^2abc(c-b)(b-a)(a-c)(a-b+c)]\sim\\
   \langle c,0,0,0,t^2abc(c-b)(b-a)(c-a)(a-b+c)\rangle,
   \end{gather*}
   \end{footnotesize}
\normalsize\noindent
which gives a cycle of length 8 in $\Lambda_{3,q}$, and a cycle of length 8 in $\Lambda_{4,q}$ if and only if $a+c=b$.

   Thus, according to the symmetry of $\Lambda_{k,q}$, we have determined all of the girth cycles in  $\Lambda_{3,q}$, and those in $\Lambda_{4,q}$ if $q> 3$.

   \begin{theorem}\label{lamda34}
   Let $\epsilon=(u_1,v_1,\ldots,u_4,v_4)$ be a tuple over $\mathbb{F}_q^*$.
      \begin{description}
         \item[(a)] In $\Lambda_{3,q}$ there is a cycle of type $\epsilon$ with $v_1+v_2=0$ if and only if (\ref{3v_12eq0}) is valid for some $r,s,t\in\mathbb{F}_q^*$.
         \item[(b)] In $\Lambda_{4,q}$ there is a cycle of type $\epsilon$ with $v_1+v_2=0$ if and only if the characteristic of $\mathbb{F}_q$ is 2 and (\ref{3v_12eq0}) is valid for some $r,s,t\in\mathbb{F}_q^*$.
         \item[(c)] In $\Lambda_{3,q}$ there is a cycle of type $\epsilon$ with $v_1+v_2\neq 0$ if and only if (\ref{03x0}--\ref{03x2}) are valid for some $t\in\mathbb{F}_q^*$ and distinct $a,b,c\in\mathbb{F}_q^*$.
         \item[(d)] In $\Lambda_{4,q}$ there is a cycle of type $\epsilon$ with $v_1+v_2\neq 0$ if and only if (\ref{03x0}--\ref{03x2}) are valid for some $t\in\mathbb{F}_q^*$ and distinct $a,b,c\in\mathbb{F}_q^*$ with $a+c=b$.

         \item [(e)] For $q>3$, $g(\Lambda_{3,q})=g(\Lambda_{4,q})=8$.
         \item[(f)] $g(\Lambda_{3,3})=8$, $g(\Lambda_{4,3})\geq 10$.
         \item[(g)] For $q\geq 3$, $g(\Lambda_{5,q})\geq 10$.
      \end{description}
   \end{theorem}

  \begin{proof}
  	 The first six results follow immediately from the argument preceding this theorem.
  	
  For the last result, though it is a corollary of the known bound given in \cite{Lazebnik95}, here we give a simple proof for it by using the results~(a) and (d).
  We assume there is a cycle of type $\epsilon$ in $\Lambda_{5,q}$. If $v_1+v_2=0$, according to (a) we see (\ref{3v_12eq0}) is vaid for some $r,s,t\in\mathbb{F}_q^*$ and then we have
  $$l_5^{(5)}=\rho_1(v_1,u_2,v_2,u_3,v_3,u_4)=rts^2\neq 0,$$
  contradicts the assumption. If $v_1+v_2\neq0$, according to (d) there are $a,b,t\in\mathbb{F}_q^{\ast}$ with $(a-b)(2a-b)\neq0$ such that
  \begin{align*}
  (v_1,v_2,v_3,v_4)&=(a,b-a,-a,a-b),\\
  (u_1,u_2,u_3,u_4)&=(-t(2a-b),-tb,t(2a-b),tb),
  \end{align*}
  and then we have
  	 $$l_5^{(5)}=\rho_1(v_1,u_2,v_2,u_3,v_3,u_4)=t^2ab(a-b)(2a-b)\neq 0,$$
  contradicts the assumption. Thus, there is no cycle of length 8 in $\Lambda_{5,q}$ and then we have
  $g(\Lambda_{5,q})\geq 10$.
  \end{proof}

\section{Girth Cycles of $\Lambda_{5,q}$}
  According to (\ref{0b0}), (\ref{0b1}) and Theorem~\ref{path}, we see that the first vertices of the walk $\Gamma$ form a cycle of type $(u_1,v_1,\ldots,u_5,v_5)$ in $\Lambda_{5,q}$ if and only if
 \begin{align}\label{sumv1-5eq0}
    y_6&=v_1+\cdots+v_5=0,\\
   x_6&=u_1+\cdots+u_5=0,\label{5l_0^6}
 \end{align}
and
 \begin{align}
l_1^{(6)}&=l_2^{(6)}=y_6l_0^{(6)}-\rho_4(u_1,v_1,\ldots,u_5,v_5)=\sum_{k=2}^5y_ku_k=0,\label{5l_1^6}\\
   l_3^{(6)}&=y_{6}l_{1}^{(6)}-\rho_3(v_1,u_2,\ldots,v_4,u_5,v_5)=\sum_{k=2}^5y_k^2u_k=0,\label{5l_3^6}\\
   l_4^{(6)}&=\rho_3(u_1,v_1,\ldots,u_4,v_4,u_5)=-\sum_{k=1}^4x_{k+1}^2v_k=0,\label{5l_4^6}\\
   l_5^{(6)}&=\rho_2(v_1,u_2,\ldots,v_4,u_5)=-\sum_{2\leq r\leq s\leq4}y_ru_rv_sx_{s+1}=0.\label{5l_5^6}
 \end{align}

 To determine all of the girth cycles in $\Lambda_{5,q}$, we assume now that $\epsilon=(u_1,v_1,\ldots,u_5,v_5)$ is a tuple over $\mathbb{F}_q^*$ satisfying (\ref{sumv1-5eq0}--\ref{5l_5^6}).

 From (\ref{5l_1^6}), (\ref{5l_3^6}) and (\ref{5l_5^6}) we have
 \begin{align*}
    &\sum_{k=2}^4y_ku_k(y_5-y_k)=0,\\
    &\sum_{k=2}^4y_ku_k\sum_{j=k}^4v_jx_{j+1}=0,
 \end{align*}
 and thus there is a $t\in \mathbb{F}_q^{\ast}$ such that
 \begin{align*}
    y_2u_2&=t\begin{vmatrix}
          y_5-y_3 & y_5-y_4\\
          v_3x_4+v_4x_5 & v_4x_5
          \end{vmatrix}=tv_3v_4u_4,\\
    y_3u_3&=-t\begin{vmatrix}
          y_5-y_2 & y_5-y_4\\
          v_2x_3+v_3x_4+v_4x_5 & v_4x_5
          \end{vmatrix}=-tv_4(v_2(u_3+u_4)+v_3u_4),\\
    y_4u_4&=t\begin{vmatrix}
          y_5-y_2 & y_5-y_3\\
          v_2x_3+v_3x_4+v_4x_5 & v_3x_4+v_4x_5
          \end{vmatrix}=tv_2(v_3u_3+v_4(u_3+u_4)),
 \end{align*}
 which can also be rewritten as
  \begin{gather}
    v_1u_2=tv_3v_4u_4,\label{y2u2}\\
    (v_{1,2}+tv_2v_4)u_3=-tv_4v_{2,3}u_4,\label{y3u3}\\
    tv_2v_{3,4}u_3=-(v_{4,5}+tv_2v_4)u_4,\label{y4u4}
 \end{gather}
 respectively, where $v_{i,j}$ denotes $v_i+v_j$ for $1\leq i< j \leq5$.

 From (\ref{y3u3}) and (\ref{y4u4}) we see
 $$(v_{1,2}+tv_2v_4)(v_{4,5}+tv_2v_4)=tv_2v_{3,4}tv_4v_{2,3}, $$
 and then, from $v_2v_4-v_{2,3}v_{3,4}=-v_3(v_2+v_3+v_4)=v_3v_{1,5}$ and $v_{1,2}+v_{4,5}=-v_3$ we see
 \begin{align}\label{1t^2}
    v_2v_3v_4v_{1,5}t^2-v_2v_3v_4t+v_{1,2}v_{4,5}=0.
 \end{align}

 \subsection{Discussion on the Case $(v_2+v_3)(v_1+v_5)=0$}

 In this subsection we deal with the case $v_{2,3}v_{1,5}=0$.

 At first, we assume $v_{2,3}=0$.
 Clearly, we have $v_{1,5}=-v_4\neq 0$. From (\ref{y3u3}) we see
 \begin{align}
    tv_2v_4+v_{1,2}=0,\label{b00}
 \end{align}
 and $v_{1,2}\neq 0$.

 If $v_{3,4}=0$, then we have $v_4=v_2$. From (\ref{y4u4}) we see $v_{4,5}+tv_2v_4=0$ and thus $v_5=v_1$, $v_3=-v_2=v_1+v_3+v_4+v_5=2v_1$, $4tv_1=1$. Then, we see $u_2+u_4=0$ from (\ref{y2u2}), $u_3+u_5=0$ from (\ref{5l_1^6}), $u_1=0$ from (\ref{5l_0^6}), contradicts  $u_1\in\mathbb{F}_q^{\ast}$.

 Hence, we must have $v_{3,4}\neq0$. Let $\beta$ be the element in $\mathbb{F}_q^{\ast}$ such that
 \begin{align}\label{41u_4}
    u_4=v_{1,2}v_{3,4}v_1v_5\beta.
 \end{align}
 Then, from (\ref{y4u4}) and (\ref{b00}) we see $\delta=v_1+v_{1,2}=-(v_{4,5}+tv_2v_4)\neq0$ and
 \begin{align}\label{41u_3}
    u_3=-v_4 v_1v_5\delta\beta.
 \end{align}
 From (\ref{y2u2}) we see
 \begin{align}\label{41u_2}
    u_2&=v_1^{-1}tv_3v_4u_4=v_{1,2}^2v_{3,4}v_5\beta.
 \end{align}
 From (\ref{5l_1^6}) and (\ref{41u_4}--\ref{41u_2}) we have
 \begin{align}
    u_5&=v_5^{-1}(v_1u_2+v_{1,2}u_3+v_1u_4)\nonumber\\
       &=(v_1v_{1,2}^2v_{3,4}-v_{1,2}v_4 v_1\delta+v_{1,2}v_{3,4}v_1^2)\beta\nonumber\\
       &=v_1v_{1,2}(v_{3,4}(v_{1,2}+v_1)-v_4\delta)\beta\nonumber\\
       &=v_1v_{1,2}v_3\delta\beta.\label{41u_5}
 \end{align}
 Then, from (\ref{5l_0^6}) we have
 \begin{align}\label{41u_1}
    u_1&=-u_2-u_3-u_4-u_5\nonumber\\
       &=-(v_{1,2}^2v_{3,4}v_5- v_4 v_1v_5\delta+v_{1,2}v_{3,4}v_1v_{5}+v_1v_{1,2}v_3\delta)\beta\nonumber\\
       &=-(v_{1,2}v_{3,4}v_5-v_4v_1v_5+v_1v_{1,2}v_3)\delta\beta\nonumber\\
       &=-\left(v_{1,2}(v_4v_5+v_3v_{1,5})-v_4v_1v_5\right)\delta\beta\nonumber\\
       &=-(v_{1,2}(v_5-v_3)-v_1v_5)v_4\delta\beta\nonumber\\
       &=-(v_5-v_3+v_1)v_2v_4\delta\beta=v_{3,4}v_2v_4\delta\beta.
 \end{align}
 From (\ref{41u_4}--\ref{41u_1}) and $v_1+v_4+v_5=-v_{2,3}=0$ we have
 \begin{align*}
    &(v_1v_2v_4\delta\beta^2)^{-1}\sum_{k=1}^4x_{k+1}^2v_k\\
    =&(v_1v_2v_4\delta\beta^2)^{-1}\left(v_1u_1^2+v_4u_5^2-v_2u_3(2u_1+2u_2+u_3)\right)\\
    =&v_2\delta(v_4v_{3,4}^2+v_1v_{1,2}^2)+v_5(2v_{3,4}v_2v_4\delta+2 v_{1,2}^2v_{3,4}v_5-v_1v_4v_5\delta)\\
    =&2v_5^2v_{1,2}^2v_{3,4}+v_2v_4v_{3,4}\delta(v_{3,4}+2v_5)-v_1\delta\left(v_3v_{1,2}^2+v_4(v_{1,2}+v_{3,4})^2\right)\\
    =&v_{3,4}\left(2v_5^2v_{1,2}^2-v_2v_4\delta(\delta+v_4)-v_1\delta(v_{1,2}^2+v_4(\delta+v_4))\right)\\
    =&v_{3,4}v_{1,2}\left(2(v_1+v_4)^2v_{1,2}-v_4\delta(\delta+v_4)-v_1\delta v_{1,2}\right)\\
    =&v_{3,4}v_{1,2}\left((2v_{1,2}-\delta)v_4^2+(4v_1v_{1,2}-\delta^2)v_4+(2v_1-\delta)v_1v_{1,2}\right)\\
    =&v_{3,4}v_{1,2}v_2(v_4^2-v_2v_4-v_1v_{1,2})\\
    =&v_{3,4}v_{1,2}v_2v_5(v_{1,2}-v_4),
 \end{align*}
 and thus, we have $v_4=v_{1,2}$ and the following lemma.

 \begin{lemma}\label{v_23eq0}
    In $\Lambda_{5,q}$ there is a cycle of type $(u_1,v_1,\ldots,u_5,v_5)$ with
    \begin{align*}
       v_2+v_3=0
    \end{align*}
    if and only if there exist $c,d,r\in\mathbb{F}_q^{\ast}$ with $c\notin\{-d,-2d\}$ such that
    \begin{gather}
       v_1=c+d,\,\,v_2=-c,\,\,v_3=c,\,\,v_4=d,\,\,v_5=-c-2d,\label{v23eq0v}\\
       u_1=cr,\,\,u_2=dr,\,\,u_3=-(c+2d)r,\,\,u_4=(c+d)r,\,\,u_5=-cr.\label{v23eq0u}
    \end{gather}
 \end{lemma}
 \begin{proof}
    "$\Longrightarrow$": From $v_2+v_3=0$, $v_4=v_1+v_2$ and $v_1+v_4+v_5=0$ we have (\ref{v23eq0v}) for some $c,d\in\mathbb{F}_q^{\ast}$ with $c\notin\{-d,-2d\}$. Furthermore, from (\ref{41u_4}--\ref{41u_1}) we see (\ref{v23eq0u}) is valid for $r=-d(c+d)(c+2d)\beta$.

    The if-part follows clearly from the argument preceding this lemma.
 \end{proof}

 Assume now $v_{1,5}=0$. Clearly, we have $v_{2,3}=-v_4\neq0$.

 From (\ref{1t^2}) we see $v_2v_3v_4t=v_{1,2}v_{4,5}$ and $v_{1,2},\,v_{4,5}\in\mathbb{F}_q^{\ast}$. Then, from (\ref{y2u2}) and (\ref{y3u3}) we see $v_2v_1u_2=v_{1,2}v_{4,5}u_4$ and $v_2v_{1,2}u_3=-v_4v_{4,5}u_4$, respectively. Let $\beta$ be the element in $\mathbb{F}_q^{\ast}$ such that
 \begin{align}\label{412u_4}
    u_4=v_1v_2v_{1,2}\beta.
 \end{align}
 Then, we have
 \begin{align}
    u_2&=v_{1,2}^2v_{4,5}\beta,\label{412u_2}\\
    u_3&=-v_1v_4v_{4,5}\beta.\label{412u_3}
 \end{align}
 From (\ref{5l_1^6}) and (\ref{412u_4}--\ref{412u_3}) we have
 \begin{align}
    u_5&=-v_1^{-1}(v_1u_2+v_{1,2}u_3-v_{4,5}u_4)\nonumber\\
       &=-(v_{1,2}-v_4-v_2)v_{1,2}v_{4,5}\beta=v_{1,2}v_{4,5}^2\beta.\label{412u_5}
 \end{align}
 Then from (\ref{5l_0^6}) we have
 \begin{align}
    u_1&=-(u_2+u_3+u_4+u_5)\nonumber\\
       &=-(v_{1,2}^2v_{4,5}-v_1v_4v_{4,5}+v_1v_2v_{1,2}+v_{1,2}v_{4,5}^2)\beta\nonumber\\
       &=-\left(v_{1,2}v_{4,5}(v_{1,2}+v_{4,5})+v_1(v_4(v_{1,2}+v_3)+v_2v_{1,2})\right)\beta\nonumber\\
       &=\left(v_{1,2}v_{4,5}+v_1(v_{1,2}-v_4)\right)v_3\beta=v_2v_4v_3\beta.\label{412u_1}
 \end{align}
 From $v_{1,5}=v_2+v_3+v_4=0$ and (\ref{412u_4}--\ref{412u_1}) we see $\beta^{-2}\sum_{k=1}^4x_{k+1}^2v_k$ is equal to
 \begin{align}
    &\beta^{-2}\left(v_1u_1^2+v_2(u_1+u_2)^2+v_3(u_4+u_5)^2+v_4u_5^2\right)\nonumber\\
    =&v_{1,2}\left((v_2v_3v_4)^2+2v_2^2v_3v_4v_{1,2}v_{4,5}+v_2v_{1,2}^3v_{4,5}^2\nonumber\right.\\
    &\left.+v_3v_{1,2}(v_1v_2+v_{4,5}^2)^2+v_4v_{1,2}v_{4,5}^4\right)\nonumber\\
    =&v_{1,2}\left((v_2v_3v_4)^2+2v_2^2v_3v_4v_{1,2}v_{4,5}+v_3v_{1,2}(v_1v_2+v_{4,5}^2)^2\nonumber\right.\\
    &\left.+(v_2v_{1,2}^2+v_4(v_{1,2}+v_3)^2)v_{1,2}v_{4,5}^2\right)\nonumber\\
    =&v_{1,2}v_3\left(v_3v_2^2v_4^2+2v_2^2v_4v_{1,2}v_{4,5}+v_{1,2}(v_1v_2+v_{4,5}^2)^2\nonumber\right.\\
    &\left.+(v_4v_3+2v_4v_{1,2}-v_{1,2}^2)v_{1,2}v_{4,5}^2\right)\nonumber\\
    =&v_{1,2}v_3\left((2v_4v_{1,2}+v_4v_3-v_{1,2}^2+2v_1v_2+v_{4,5}^2)v_{1,2}v_{4,5}^2\nonumber\right.\\
    &\left.+2v_2^2v_4v_{1,2}v_{4,5}+(v_1^2v_{1,2}-v_{2,4}v_4^2)v_2^2\right)\nonumber\\
    =&v_{1,2}v_3v_{4,5}\left((2v_4v_{1,2}-v_4v_{2,4}-v_1^2-v_2^2+(v_4-v_1)^2)v_{1,2}v_{4,5}\nonumber\right.\\
    &\left.+2v_2^2v_4v_{1,2}-((v_1^2+v_1v_4+v_4^2)+v_2(v_1+v_4))v_2^2\right)\nonumber\\
    =&v_{1,2}v_3v_{4,5}v_2\left((v_4-v_2)v_{1,2}(v_4-v_1)+2v_2v_4v_{1,2}-(v_1^2+v_{2,4}v_1+v_4v_{2,4})v_2\right)\nonumber\\
    =&v_{1,2}v_3v_{4,5}v_2\left(-v_4v_1^2+v_4(v_4-v_2)v_1\right)\nonumber\\
    =&v_{1,2}v_3v_{4,5}v_2v_4v_1(v_4-v_{1,2})\nonumber
 \end{align}
 and thus, we have $v_4=v_{1,2}$ and the following lemma.

 \begin{lemma}\label{v_15eq0}
    In $\Lambda_{5,q}$ there is a cycle of type $(u_1,v_1,\ldots,u_5,v_5)$ with
    \begin{align*}
       v_1+v_5=0
    \end{align*}
    if and only if there exist $b,c,r\in\mathbb{F}_q^{\ast}$ with $c\notin\{-b,-2b\}$ such that
    \begin{gather}
       v_1=-2b-c,\,\,v_2=b,\,\,v_3=c,\,\,v_4=-b-c,\,\,v_5=2b+c,\label{v_15eq0v}\\
       u_1=cr,\,\,u_2=-(b+c)r,\,\,u_3=(2b+c)r,\,\,u_4=-(2b+c)r,\,\,u_5=br.\label{v_15eq0u}
    \end{gather}
 \end{lemma}
 \begin{proof}
    "$\Longrightarrow$": From $v_2+v_3+v_4=0$, $v_4=v_1+v_2$ and $v_1+v_5=0$ we have (\ref{v_15eq0v}) for some $b,c\in\mathbb{F}_q^{\ast}$ with $c\notin\{-b,-2b\}$. Furthermore, from (\ref{412u_4}--\ref{412u_1}) we see (\ref{v_15eq0u}) is valid for $r=-b(b+c)\beta$.

    The if-part follows clearly from the argument preceding this lemma.
 \end{proof}

 \subsection{Discussion on the Case $(v_2+v_3)(v_1+v_5)\neq0$}
 In this subsection we deal with the case $v_{2,3}v_{1,5}\neq 0$.

 Let $\beta$ be the element in $\mathbb{F}_q^{\ast}$ such that
 \begin{align}\label{42u_3}
    u_3=-tv_4v_{2,3}v_1v_{1,5}v_5\beta.
 \end{align}
 From (\ref{y3u3}), we see
 \begin{align}\label{42u_4}
    u_4=(tv_2v_4+v_{1,2})v_1v_{1,5}v_5\beta.
 \end{align}
 From (\ref{y2u2}) and (\ref{1t^2}), we have
 \begin{align}
    u_2&=tv_3v_4(tv_2v_4+v_{1,2})v_{1,5}v_5\beta\nonumber\\
       &=(t^2v_2v_3v_4v_{1,5}+tv_3v_{1,2}v_{1,5})v_4v_5\beta\nonumber\\
       &=(tv_3(v_2v_4+v_{1,2}v_{1,5})-v_{1,2}v_{4,5})v_4v_5\beta.\label{42u_2}
 \end{align}
 Then, from (\ref{5l_1^6}) we see
 \begin{align}
    u_5=&v_5^{-1}(v_1u_2+v_{1,2}u_3-v_{4,5}u_4)\nonumber\\
       =&v_1\left(tv_3(v_2v_4+v_{1,2}v_{1,5})-v_{1,2}v_{4,5}\right)v_4\beta-v_{1,2}tv_4v_{2,3}v_1v_{1,5}\beta\nonumber\\
        &-v_{4,5}(tv_2v_4+v_{1,2})v_1v_{1,5}\beta\nonumber\\
       =&\left(tv_4(v_3v_2v_4+v_3v_{1,2}v_{1,5}-v_{1,2}v_{2,3}v_{1,5}-v_{4,5}v_2v_{1,5})\nonumber\right.\\
        &\left.-v_{1,2}v_{4,5}(v_4+v_{1,5})\right)v_1\beta\nonumber\\
       =&\left(tv_4v_2(v_3v_4-v_{1,2}v_{1,5}-v_{4,5}v_{1,5})+v_{1,2}v_{4,5}v_{2,3}\right)v_1\beta\nonumber\\
       =&\left(tv_4v_2v_3(v_4+v_{1,5})+v_{1,2}v_{4,5}v_{2,3}\right)v_1\beta\nonumber\\
       =&-(tv_2v_3v_4-v_{1,2}v_{4,5})v_1v_{2,3}\beta,\label{42u_5}
 \end{align}
 Furthermore, from (\ref{5l_0^6}) we have
 \begin{align}
    u_1=&-u_2-u_3-u_4-u_5\nonumber\\
       =&-(tv_3(v_2v_4+v_{1,2}v_{1,5})-v_{1,2}v_{4,5})v_4v_5\beta+tv_4v_{2,3}v_1v_{1,5}v_5\beta\nonumber\\
        &-(tv_2v_4+v_{1,2})v_1v_{1,5}v_5\beta+(tv_4v_2v_3-v_{1,2}v_{4,5})v_1v_{2,3}\beta\nonumber\\
       =&t(-v_3v_4v_5(v_2v_4+v_{1,2}v_{1,5})+v_3v_4v_1v_{1,5}v_5+v_2v_3v_4v_1v_{2,3})\beta\nonumber\\
        &+v_{1,2}(v_{4,5}v_4v_5-v_1v_{1,5}v_5-v_{4,5}v_1v_{2,3})\beta\nonumber\\
       =&t(-v_2v_3v_4^2v_5-v_3v_4v_5v_{1,5}v_2+v_2v_3v_4v_1v_{2,3})\beta\nonumber\\
        &+v_{1,2}(v_{4,5}v_4v_5-v_1(v_5(v_{1,5}+v_{2,3})+v_4v_{2,3}))\beta\nonumber\\
       =&t(v_2v_3v_4v_5v_{2,3}+v_2v_3v_4v_1v_{2,3})\beta+v_{1,2}v_4(v_{4,5}v_5-v_1(v_{2,3}-v_5))\beta\nonumber\\
       =&tv_2v_3v_4v_{2,3}v_{1,5}\beta+v_{1,2}v_4(v_5(v_1+v_{4,5})-v_1v_{2,3})\beta\nonumber\\
       =&(tv_2v_3-v_{1,2})v_4v_{2,3}v_{1,5}\beta.\label{42u_1}
 \end{align}
 Then, from
 \begin{align*}
    u_1+u_2&=\big((tv_2v_3-v_{1,2})v_{2,3}v_{1,5}+(tv_3(v_2v_4+v_{1,2}v_{1,5})-v_{1,2}v_{4,5})v_5\big)v_4\beta\\
           &=\big(tv_3(v_2v_{2,3}v_{1,5}+v_2v_4v_5+v_{1,2}v_{1,5}v_5)-v_{1,2}(v_{2,3}v_{1,5}+v_{4,5}v_5)\big)v_4\beta\\
           &=\big(tv_3(v_2(v_{2,3}v_1-v_{1,5}v_5)+v_{1,2}v_{1,5}v_5)-v_{1,2}(v_{2,3}v_1-v_5v_1)\big)v_4\beta\\
           &=\big(tv_3(v_2v_{2,3}+v_{1,5}v_5)+v_{1,2}(v_5-v_{2,3})\big)v_1v_4\beta
 \end{align*}
 and
 \begin{align*}
    u_1+u_2+u_3&=(u_1+u_2)-tv_4v_{2,3}v_1v_{1,5}v_5\beta\\
               &=\big(tv_2(v_3v_{2,3}-v_{1,5}v_5)+v_{1,2}(v_5-v_{2,3})\big)v_1v_4\beta
 \end{align*}
 we see
 \begin{align}
    &\beta^{-2}\sum_{k=1}^4x_{k+1}^2v_k=\beta^{-2}\left(v_4u_5^2+v_1u_1^2+v_2(u_1+u_2)^2+v_3(u_1+u_2+u_3)^2\right)\nonumber\\
    =&v_4(tv_2v_3v_4-v_{1,2}v_{4,5})^2v_1^2v_{2,3}^2+v_1(tv_2v_3-v_{1,2})^2v_4^2v_{2,3}^2v_{1,5}^2+\nonumber\\
      &\ \ \ \ \ v_2\big(tv_3(v_2v_{2,3}+v_5v_{1,5})+v_{1,2}(v_5-v_{2,3})\big)^2v_1^2v_4^2+\nonumber\\
      &\ \ \ \ \ v_3\big(tv_2(v_3v_{2,3}-v_5v_{1,5})+v_{1,2}(v_5-v_{2,3})\big)^2v_1^2v_4^2\nonumber\\
    =&v_1v_2v_3v_4^2v_{2,3}\xi_2t^2+2v_1v_2v_3v_4^2v_{2,3}^2\xi_1t+v_1v_4v_{2,3}\xi_0,\label{transl_4^6}
 \end{align}
 where
 \begin{align*}
    \xi_2&=v_1v_2v_3v_4v_{2,3}+v_2v_3v_{2,3}v_{1,5}^2+v_1v_2v_3v_{2,3}^2+v_1v_5^2v_{1,5}^2\\
         &=-v_1v_2v_3v_{2,3}v_{1,5}+v_2v_3v_{2,3}v_{1,5}^2+v_1v_5^2v_{1,5}^2\\
         &=v_5v_{1,5}(v_1v_5v_{1,5}+v_2v_3v_{2,3}),\\
    \xi_1&=-v_1v_{1,2}v_{4,5}-v_{1,2}v_{1,5}^2+v_1v_{1,2}(v_5-v_{2,3})\\
         &=v_1v_{1,2}v_{1,5}-v_{1,2}v_{1,5}^2=-v_5v_{1,2}v_{1,5},\\
    \xi_0&=v_1v_{2,3}v_{1,2}^2v_{4,5}^2+v_4v_{2,3}v_{1,2}^2v_{1,5}^2+v_1v_4v_{1,2}^2(v_5-v_{2,3})^2\\
         &=v_{1,2}^2\big(v_1v_{2,3}v_{4,5}^2+v_4v_{2,3}v_{1,5}^2+v_1v_4(v_{4,5}+v_{1,5})^2\big)\\
         &=v_{1,2}^2v_{4,5}v_{1,5}(2v_1v_4-v_1v_{4,5}-v_4v_{1,5})\\
         &=-v_{1,2}^2 v_{4,5}v_{1,5}v_5v_{1,4}.
 \end{align*}
 Hence from $v_{2,3}v_{1,5}\neq0$, (\ref{5l_4^6}) and (\ref{transl_4^6}) we have
 \begin{align}\label{3t^2}
    v_2v_3v_4(v_1v_5v_{1,5}+v_2v_3v_{2,3})t^2-2v_2v_3v_4v_{2,3}v_{1,2}t-v_{1,2}^2v_{4,5}v_{1,4}=0.
 \end{align}

 If $v_{1,2}=0$, from (\ref{1t^2}) and (\ref{3t^2}), we have $t=v_{1,5}^{-1}$ and $v_1v_5v_{1,5}+v_2v_3v_{2,3}=0$, the later equality is equivalent to $v_3=v_{1,5}$.
 \begin{lemma}\label{v_12eq0}
    In $\Lambda_{5,q}$ there is a cycle of type $(u_1,v_1,\ldots,u_5,v_5)$ with
    \begin{align*}
       (v_2+v_3)(v_1+v_5)\neq0,\,v_1+v_2=0,
    \end{align*}
    if and only if there exist $b,c,r\in\mathbb{F}_q^{\ast}$ with $b\notin \{-c,-2c\}$ such that
    \begin{gather}
       v_1=-b,\,v_2=b,\,v_3=c,\,v_4=-b-2c,\,v_5=b+c,\label{v_12eq0v}\\
       u_1=cr,\,u_2=-(b+2c)r,\,u_3=(b+c)r,\,u_4=-br,\,u_5=br.\label{v_12eq0u}
    \end{gather}
 \end{lemma}
 \begin{proof}
    "$\Longrightarrow$": From $v_1+v_2=0$, $v_3=v_1+v_5$ and $v_3+v_4+v_5=0$ we have (\ref{v_12eq0v}) for some  $b,c\in\mathbb{F}_q^{\ast}$ with $b\notin \{-c,-2c\}$. Furthermore, from (\ref{42u_3}--\ref{42u_1}) we see (\ref{v_12eq0u}) is valid for $r=-b(b+c)(b+2c)\beta$.

    The if-part follows clearly from the argument preceding this lemma.
 \end{proof}

 Now we assume further $v_{1,2}\neq0$. From $v_{1,4}v_{1,2}\times$(\ref{1t^2})$+$(\ref{3t^2}), we get
 \begin{align}\label{1t}
    &(v_1v_5v_{1,5}+v_2v_3v_{2,3}+v_{1,4}v_{1,2}v_{1,5})t=v_{1,2}(2v_{2,3}+v_{1,4}).
 \end{align}

 If $2v_{2,3}+v_{1,4}=0$, i.e. $v_5=v_{2,3}$, from (\ref{1t}) we see
 \begin{align*}
    &v_1v_5v_{1,5}+v_2v_3v_{2,3}+v_{1,4}v_{1,2}v_{1,5}\\
    =&v_1v_{2,3}v_{1,2,3}+v_2v_3v_{2,3}-2v_{2,3}v_{1,2}v_{1,2,3}\\
    =&v_{2,3}(v_1v_{1,2}+v_3v_{1,2}-2v_{1,2}v_{1,2,3})\\
    =&-v_{2,3}v_{1,2}(v_1+v_3+2v_2)=0,
 \end{align*}
 and then $v_1=-v_3-2v_2, \,v_4=-v_1-2v_{2,3}=-v_3$.

 \begin{lemma}\label{v_23eqv_5}
    In $\Lambda_{5,q}$ there is a cycle of type $(u_1,v_1,\ldots,u_5,v_5)$ with
    \begin{align}\label{23eq5}
       (v_1+v_5)(v_1+v_2)\neq0,\, v_5=v_2+v_3,
    \end{align}
    if and only if there exist $b,c,r\in\mathbb{F}_q^{\ast}$ with $c\notin \{-b,-2b\}$ such that
    \begin{gather}
       v_1=-2b-c,\,v_2=b,\,v_3=c,\,v_4=-c,\,v_5=b+c,\label{23eq5v}\\
       u_1=cr,\,u_2=-cr,\,u_3=(b+c)r,\,u_4=-(2b+c)r,\,u_5=br.\label{23eq5u}
    \end{gather}
 \end{lemma}
 \begin{proof}
    "$\Longrightarrow$": From $v_5=v_2+v_3$, $v_1=-v_3-2v_2$ and $v_4=-v_3$ we see (\ref{23eq5v}) for some $b,c\in\mathbb{F}_q^{\ast}$ with $c\notin \{-b,-2b\}$. From (\ref{42u_4}) we see $tv_2v_4+v_{1,2}=-tbc-b-c\neq0$ and thus from (\ref{1t^2}) and
    \begin{align}
       &v_2v_3v_4v_{1,5}t^2-v_2v_3v_4t+v_{1,2}v_{4,5}\nonumber\\
       =&b^2c^2t^2+bc^2t-(b+c)b=b(ct-1)(bct+b+c)\label{dett}
    \end{align}
    we see $t=c^{-1}$. Then, from (\ref{42u_3}--\ref{42u_1}), we see (\ref{23eq5u}) is valid for $r=b(b+c)(2b+c)\beta$.

    "$\Longleftarrow$": For $b,c,r\in\mathbb{F}_q^{\ast}$ with $c\notin \{-b,-2b\}$, let $(u_1,v_1,\ldots,u_5,v_5)$ be the tuple defined by (\ref{23eq5v}) and (\ref{23eq5u}). Clearly, we have (\ref{23eq5}). By setting $\beta=r(b(b+c)(2b+c))^{-1}$ and $t=c^{-1}$, one can check (\ref{42u_3}--\ref{42u_1}) easily. From (\ref{dett}), $2v_{2,3}=2(b+c)=-v_{1,4}$ and $v_1v_5v_{1,5}+v_2v_3v_{2,3}=(-2b-c)(b+c)(-b)+bc(b+c)=2b(b+c)^2=-v_{1,4}v_{1,2}v_{1,5}$, we see (\ref{1t^2}) and (\ref{3t^2}) are valid. Then, one can conclude that in $\Lambda_{5,q}$ there is a cycle of type $(u_1,v_1,\ldots,u_5,v_5)$ with (\ref{23eq5}).
 \end{proof}

 Now we assume $v_5\neq v_2+v_3$ further. From (\ref{1t}) we see
 \begin{gather}
    v_1v_5v_{1,5}+v_2v_3v_{2,3}+v_{1,4}v_{1,2}v_{1,5}\neq0,\\
    t=(v_1v_5v_{1,5}+v_2v_3v_{2,3}+v_{1,4}v_{1,2}v_{1,5})^{-1}v_{1,2}(2v_{2,3}+v_{1,4}).\label{1ttrans}
 \end{gather}

 Below we consider to deduce a simple condition for (\ref{1ttrans}) satisfying both (\ref{1t^2}) and (\ref{3t^2}). From $v_{1,5}\times$(\ref{3t^2})$-(v_1v_5v_{1,5}+v_2v_3v_{2,3})\times$(\ref{1t^2}) we see
 \begin{align}
    &(v_1v_5v_{1,5}+v_2v_3v_{2,3}-2v_{1,2}v_{2,3}v_{1,5})v_2v_3v_4t\nonumber\\
    =&v_{1,4}v_{4,5}v_{1,2}^2v_{1,5}+v_{1,2}v_{4,5}(v_1v_5v_{1,5}+v_2v_3v_{2,3})\nonumber\\
    =&v_{1,2}v_{4,5}(v_{1,4}v_{1,2}v_{1,5}+v_1v_5v_{1,5}+v_2v_3v_{2,3}).\label{2t}
 \end{align}
 From $v_2v_3v_4\times$(\ref{1t})$-$(\ref{2t}) we see
 \begin{align*}
    &(2v_{2,3}+v_{1,4})v_{1,2}v_{1,5}v_2v_3v_4t\\
    =&v_{1,2}\left((v_{2,3}-v_5)v_2v_3v_4-v_{4,5}(v_{1,4}v_{1,2}v_{1,5}+v_1v_5v_{1,5}+v_2v_3v_{2,3})\right)\\
    =&v_{1,2}\left(v_2v_3(v_{2,3}v_4-v_5v_4-v_{4,5}v_{2,3})-v_{4,5}v_{1,5}(v_{1,4}v_{1,2}+v_1v_5)\right)\\
    =&v_{1,2}v_{1,5}\left(v_2v_3v_5-v_{4,5}(v_{1,4}v_{1,2}+v_1v_5)\right),
 \end{align*}
 and thus from $v_{1,2}v_{1,5}\neq0$ we see
 \begin{align}\label{3t}
    (2v_{2,3}+v_{1,4})v_2v_3v_4t=v_2v_3v_5-v_{4,5}(v_{1,4}v_{1,2}+v_1v_5).
 \end{align}
 Then, from (\ref{1t}) and (\ref{3t}) we see
 \begin{align}
    &\left(v_2v_3v_{2,3}+(v_{1,4}v_{1,2}+v_1v_5)v_{1,5}\right)\left(v_2v_3v_5-(v_{1,4}v_{1,2}+v_1v_5)v_{4,5}\right)\nonumber\\
    =&(2v_{2,3}+v_{1,4})^2v_{1,2}v_2v_3v_4.\label{relationv}
 \end{align}
 From $v_{1,4}v_{1,2}+v_1v_5=v_2v_4-v_1v_3$, we see the difference of the two sides of (\ref{relationv}) is
 \begin{align*}
    &\left(v_2v_3v_{2,3}-(v_2v_4-v_1v_3)v_{2,3,4}\right)\left((v_2v_4-v_1v_3)v_{1,2,3}-v_2v_3v_{1,2,3,4}\right)\nonumber\\
    &-(2v_{2,3}+v_{1,4})^2v_{1,2}v_2v_3v_4=\sigma_3v_4^3+\sigma_2v_4^2+\sigma_1v_4+\sigma_0,
 \end{align*}
 where $v_{2,3,4}=v_{2,3}+v_4$, $v_{1,2,3}=v_1+v_{2,3}$, $v_{1,2,3,4}=v_{1,2,3}+v_4$ and
 \begin{align*}
    \sigma_3&=(-v_2)(v_2v_{1,2})-v_{1,2}v_2v_3=-v_2v_{1,2}v_{2,3},\\
    \sigma_2&=(-v_2)(-v_3v_{1,2,3}v_{1,2})+(-v_2v_{2,3}+v_1v_3)(v_2v_{1,2})-2(2v_{2,3}+v_1)v_{1,2}v_2v_3\\
            &=v_2v_{1,2}(v_3v_{1,2,3}-v_2v_{2,3}+v_1v_3-2(2v_{2,3}+v_1)v_3)\\
            &=-v_2v_{1,2}v_{2,3}(3v_3+v_2),\\
    \sigma_1&=(-v_2v_{2,3}+v_1v_3)(-v_3v_{1,2,3}v_{1,2})+(v_3v_{2,3}v_{1,2})(v_2v_{1,2})-(2v_{2,3}+v_1)^2v_{1,2}v_2v_3\\
            &=-v_{1,2}v_3\left((-v_2v_{2,3}+v_1v_3)v_{1,2,3}-v_{2,3}v_{1,2}v_2+(2v_{2,3}+v_1)^2v_2\right)\\
            &=-v_{1,2}v_3v_{2,3}(-v_2v_{1,2,3}+v_1^2+v_1v_3-v_{1,2}v_2+4v_{2,3}v_2+4v_1v_2)\\
            &=-v_{1,2}v_3v_{2,3}(v_1v_{1,3}-v_{1,2}v_2+3v_{1,2,3}v_2),\\
            &=-v_{1,2}v_3v_{2,3}(v_{1,2}^2+v_3v_{1,2}+2v_2v_3+v_2^2),\\
    \sigma_0&=(v_3v_{2,3}v_{1,2})(-v_3v_{1,2,3}v_{1,2})=-v_3^2v_{1,2}^2v_{2,3}v_{1,2,3}.
 \end{align*}
 Then, from $v_{1,2}v_{2,3}\neq0$ and (\ref{relationv}) we see
 \begin{align*}
    v_2v_4^3+v_2(3v_3+v_2)v_4^2+v_3(v_{1,2}^2+v_3v_{1,2}+2v_2v_3+v_2^2)v_4+v_3^2v_{1,2}v_{1,2,3}=0
 \end{align*}
 i.e.
 \begin{align}\label{relav1-4}
    (v_3v_{1,2}^2+v_3^2v_{1,2}+v_2v_4(v_2+2v_3+v_4))v_{3,4}=0.
 \end{align}
 From (\ref{42u_1}) and (\ref{3t}) we see
 \begin{align*}
    &(2v_{2,3}+v_{1,4})(v_{2,3}v_{1,5}\beta)^{-1}u_1\\
    =&(2v_{2,3}+v_{1,4})(tv_2v_3-v_{1,2})v_4\\
    =&v_2v_3v_5-v_{4,5}(v_{1,4}v_{1,2}+v_1v_5)-v_{1,2}v_4(2v_{2,3}+v_{1,4})\\
    =&(v_2v_3-v_{4,5}v_1)v_5-(v_{4,5}-v_4)v_{1,4}v_{1,2}-2(v_{2,3}+v_{1,4})v_{1,2}v_4\\
    =&(v_2v_3-v_{4,5}v_1-v_{1,4}v_{1,2}+2v_{1,2}v_4)v_5\\
    =&(v_2v_3-v_{4,5}v_1-v_{1,2}v_1+v_{1,2}v_4)v_5\\
    =&(v_2v_3+v_1v_3+v_{1,2}v_4)v_5=v_{1,2}v_{3,4}v_5,
 \end{align*}
 and thus we have $v_{3,4}\neq0$. Then, from (\ref{relav1-4}) we have
 \begin{align}\label{relav}
    v_3v_{1,2}^2+v_3^2v_{1,2}+v_2v_4(v_{2,3}+v_{3,4})=0.
 \end{align}
 and thus
 \begin{align}
     &v_2v_3v_{2,3}+v_{1,5}(v_{1,4}v_{1,2}+v_1v_5)\nonumber\\
    =&v_2v_3v_{2,3}-(v_{2,3}+v_4)(v_4v_2-v_1v_3)\nonumber\\
    =&v_3v_{2,3}(v_2+v_1)+(v_1v_3-v_{2,3}v_2)v_4-v_2v_4^2\nonumber\\
    =&v_{3}v_{1,2}(v_{2,3}+v_{1,2}+v_3)+(v_1v_3-v_{2,3}v_2+v_2(v_2+2v_3))v_4\nonumber\\
    =&v_3v_{1,2}(v_1+2v_{2,3})+(v_1v_3+v_2v_3)v_4\nonumber\\
    =&v_3v_{1,2}(v_1+2v_{2,3}+v_4)\nonumber\\
    =&v_3v_{1,2}(2v_{2,3}+v_{1,4}).\label{relavtrans}
 \end{align}
 Then, from (\ref{1ttrans}) we see $t=v_3^{-1}$ and then, from (\ref{relav}) and (\ref{42u_3}--\ref{42u_1}) we have
 \begin{align}
    u_1&=(tv_2v_3-v_{1,2})v_4v_{2,3}v_{1,5}\beta=-v_1v_4v_{2,3}v_{1,5}\beta,\label{42finalu_1}\\
    u_2&=(tv_3(v_2v_4+v_{1,2}v_{1,5})-v_{1,2}v_{4,5})v_4v_5\beta\nonumber\\
       &=(v_2v_4+v_{1,2}(v_1-v_4))v_4v_5\beta\nonumber\\
       &=-(v_{1,2}-v_4)(v_{1,2}+v_{3,4})v_1v_4\beta\nonumber\\
       &=-(v_{1,2}^2+v_3v_{1,2}-v_4v_{3,4})v_1v_4\beta\nonumber\\
       &=v_3^{-1}(v_2(v_{2,3}+v_{3,4})+v_3v_{3,4})v_1v_4^2\beta\nonumber\\
       &=v_3^{-1}v_{2,3}(v_2+v_{3,4})v_1v_4^2\beta\nonumber\\
       &=-v_3^{-1}v_{2,3}v_{1,5}v_1v_4^2\beta,\label{42finalu_2}\\
    u_3&=-tv_4v_{2,3}v_1v_{1,5}v_5\beta=-v_3^{-1}v_4v_{2,3}v_1v_{1,5}v_5\beta,\label{42finalu_3}\\
    u_4&=(tv_2v_4+v_{1,2})v_1v_{1,5}v_5\beta\nonumber\\
       &=-v_3^{-1}(v_2v_4+v_3v_{1,2})(v_{1,2,3}+v_4)v_1v_{1,5}\beta\nonumber\\
       &=-v_3^{-1}(v_3v_{1,2}v_{1,2,3}+v_4(v_2v_{1,2,3}+v_2v_4+v_3v_{1,2}))v_1v_{1,5}\beta\nonumber\\
       &=-v_3^{-1}(-v_2(v_2+2v_3+v_4)+(v_{2,3}v_{1,2}+v_2v_{3,4}))v_4v_1v_{1,5}\beta\nonumber\\
       &=-v_3^{-1}v_{2,3}v_4v_1^2v_{1,5}\beta,\label{42finalu_4}\\
    u_5&=-(tv_2v_3v_4-v_{1,2}v_{4,5})v_1v_{2,3}\beta\nonumber\\
       &=-(v_2v_4+v_{1,2}(v_{1,2}+v_3))v_1v_{2,3}\beta\nonumber\\
       &=-v_3^{-1}(v_3v_2v_4-v_2v_4(v_{2,3}+v_{3,4}))v_1v_{2,3}\beta\nonumber\\
       &=-v_3^{-1}v_2v_4v_{1,5}v_1v_{2,3}\beta\label{42finalu_5}.
 \end{align}
 Let $r=-v_3^{-1}v_1v_4v_{2,3}v_{1,5}\beta$. From (\ref{42finalu_1}--\ref{42finalu_5}) we see
 \begin{align}
    u_1=v_3r,\,\,u_2=v_4r,\,\,u_3=v_5r,\,\,u_4=v_1r,\,\,u_5=v_2r,
 \end{align}
 and thus, we have the following lemma.

 \begin{lemma}\label{v_finial}
    In $\Lambda_{5,q}$ there is a cycle of type $(u_1,v_1,\ldots,u_5,v_5)$ with
    \begin{align}\label{condition}
       v_{2,3}v_{1,2}v_{1,5}(v_{2,3}-v_5)\neq0,
    \end{align}
    if and only if there are $b,c,d,r\in\mathbb{F}_q^{\ast}$ with
    \begin{align}
       (b+c)(c+d)(b+c+d)\neq0
    \end{align}
    such that
    \begin{align}\label{sufficient2}
       f_{b,c,d}(x)=cx^2+c(2b+c)x+b(c+d)(b+c+d)
    \end{align}
    is reducible in $\mathbb{F}_q$ and, for $a\in\mathbb{F}_q^{\ast}\backslash \{-b,-2b-2c-d\}$ with $f_{b,c,d}(a)=0$,
    \begin{gather}
       v_1=a,\,\,v_2=b,\,\,v_3=c,\,\,v_4=d,\,\,v_5=-(a+b+c+d),\label{5qv}\\
       u_1=cr,\,\,u_2=dr,\,\,u_3=-(a+b+c+d)r,\,\,u_4=ar,\,\,u_5=br.\label{5qu}
    \end{gather}
 \end{lemma}
 \begin{proof}
   The proof follows clearly from the argument preceding this lemma.
 \end{proof}

\subsection{Girth and Girth Cycles of $\Lambda_{5,q}$}
 \begin{theorem}\label{cyclein5q}
    In $\Lambda_{5,q}$ there is a cycle of type $(u_1,v_1,\ldots,u_5,v_5)$ if and only if (\ref{5qv}) and (\ref{5qu}) are valid for some $a,b,c,d,r\in\mathbb{F}_q^{\ast}$ with $f_{b,c,d}(a)=0$ and $a+b+c+d\neq0$.
 \end{theorem}
 \begin{proof}
    The only-if-part can be checked easily according to Lemmas \ref{v_23eq0}--\ref{v_finial}.

    To show the if-part, we assume that (\ref{5qv}) and (\ref{5qu}) are valid for some $a,b,c,d,r\in\mathbb{F}_q^{\ast}$ with $f_{b,c,d}(a)=0$ and $a+b+c+d\neq0$.

    If $b+c=0$, from $a+d=a+b+c+d\neq0$ and
    $$f_{b,c,d}(a)=ca^2-c^2a-c(c+d)d=c(a+d)(a-c-d)=0$$
    we see $a=c+d\neq0$, $c+2d=a+d\neq0$ and thus according to Lemma \ref{v_23eq0} in $\Lambda_{5,q}$ there is a cycle of type $(u_1,v_1,\ldots,u_5,v_5)$.

    If $b+c+d=0$, from
    $$f_{b,c,d}(a)=ca^2+c(2b+c)a=ca(a+2b+c)=0$$
    we see $a=-2b-c\neq0$, $b+c=-d\neq0$ and thus according to Lemma \ref{v_15eq0} in $\Lambda_{5,q}$ there is a cycle of type $(u_1,v_1,\ldots,u_5,v_5)$.

    If $a+b=0$, from
    $$f_{b,c,d}(a)=cb^2-c(2b+c)b+b(c+d)(b+c+d)=bd(b+2c+d)=0$$
    we see $d=-b-2c\neq0$, $b+c=-c-d=-(a+b+c+d)\neq0$ and thus according to Lemma \ref{v_12eq0} in $\Lambda_{5,q}$ there is a cycle of type $(u_1,v_1,\ldots,u_5,v_5)$.

    If $c+d=0$, from
    $$f_{b,c,d}(a)=ca^2+c(2b+c)a=ca(a+2b+c)=0$$
    we see $a=-2b-c\neq0$, $b+c=-a-b=-(a+b+c+d)\neq0$  and thus according to Lemma \ref{v_23eqv_5} in $\Lambda_{5,q}$ there is a cycle of type $(u_1,v_1,\ldots,u_5,v_5)$.

    Assume now $(b+c)(b+c+d)(a+b)(c+d)\neq0$. If $b+2c+d=0$, we have $a+2b+2c+d=a+b\neq0$. If $b+2c+d\neq0$, from $f_{b,c,d}(a)=0$ and
    \begin{align*}
       &f_{b,c,d}(-2b-2c-d)\\
       =&c(2b+2c+d)^2-c(2b+c)(2b+2c+d)+b(c+d)(b+c+d)\\
       =&(c+d)(b+c)(b+2c+d)\neq0
    \end{align*}
    we also have $a+2b+2c+d\neq0$. Hence, according to Lemma \ref{v_finial} in $\Lambda_{5,q}$ there is a cycle of type $(u_1,v_1,\ldots,u_5,v_5)$.
 \end{proof}

 \begin{corollary}
   \begin{description}
      \item[(a)] For $q>3$, $g(\Lambda_{5,q})=10$.
      \item[(b)] $g(\Lambda_{5,3})\geq12$.
   \end{description}
 \end{corollary}

 \begin{proof}
     \textbf{(a)}
     If $q$ is a prime power greater than 3, then for $c\in\mathbb{F}_q^{\ast}\backslash \{1,2\}$, according to Lemma \ref{v_23eq0} the graph $\Lambda_{5,q}$ has a cycle of type
          $$(c,c-1,-1,-c,2-c,c,c-1,-1,-c,2-c).$$
     Hence, we have $g(\Lambda_{5,q})=10$.

     \textbf{(b)}
     For any $b,c,d\in\mathbb{F}_3^{\ast}$, we have $\{b,2b\}=\mathbb{F}_3^{\ast}$ and $(b+c)(c+d)(b+c+d)=0$, and thus from Lemmas~\ref{v_23eq0}--\ref{v_finial} we see $\Lambda_{5,3}$ has no cycle of length 10, this means $g(\Lambda_{5,3})\geq12$.
 \end{proof}

\section{Girth Cycles of $\Lambda_{k,3}$, $4\leq k\leq 8$}
 We note that $\mathbb{F}_3^{\ast}=\{1,2\}$ and $a^2=1$ for any $a\in\mathbb{F}_3^{\ast}$.

 At first we conclude that in $\Lambda_{3,3}$ there is no cycle of length 10 and thus $g(\Lambda_{k,3})\geq12$ for $k\geq4$. Assume in contrast that $\Lambda_{3,3}$ has a cycle of type $(u_1,v_1,\ldots,u_5,v_5)$. Then we have (\ref{sumv1-5eq0}--\ref{5l_3^6}). Since $y_2=v_1$ and  $y_5=-v_5$ are nonzero, from (\ref{5l_0^6}--\ref{5l_3^6}) we see $y_3y_4=0$.

 If $y_3=0$, then $y_4=v_3\neq0$ and thus from (\ref{5l_3^6}) we have $u_2+u_4+u_5=0$, which implies $u_2=u_4=u_5$. Hence, from (\ref{5l_1^6}) we see $y_2=y_4=y_5$, contradicts $y_5-y_4=v_4\neq0$.

 If $y_4=0$, then $y_3=-v_3\neq0$ and thus from (\ref{5l_3^6}) we have $u_2+u_3+u_5=0$, which implies $u_2=u_3=u_5$. Hence, from (\ref{5l_1^6}) we see $y_2=y_3=y_5$, contradicts $y_3-y_2=v_2\neq0$.

 Now we consider to determine all the cycles of length 12 in $\Lambda_{3,3}$. Assume in $\Lambda_{3,3}$ there is a cycle of type $\epsilon=(u_1,v_1,\ldots,u_6,v_6)$. Then, according to Theorem \ref{path} we have
 \begin{gather}
    v_1+v_2+v_3+v_4+v_5+v_6=0,\label{sumv16eq0}\\
    u_1+u_2+u_3+u_4+u_5+u_6=0,\label{6l_0^7}\\
    y_2u_2+y_3u_3+y_4u_4+y_5u_5+y_6u_6=0,\label{6l_1^7}\\
    y_2^2u_2+y_3^2u_3+y_4^2u_4+y_5^2u_5+y_6^2u_6=0.\label{6l_3^7}
 \end{gather}
 From $y_2y_6=-v_1v_6\neq0$, (\ref{6l_0^7}) and (\ref{6l_3^7}), we see $y_3y_4y_5=0$. We note that
 \begin{align}\label{y_ineqy_i+1}
   y_i\neq y_{i+1},\,1\leq i\leq 5.
 \end{align}

 If two of $y_3, y_4, y_5$ are zero, then we have $y_3=y_5=0$ and $y_4\neq0$. From (\ref{6l_3^7}) we see $u_2+u_4+u_6=0$, i.e. $u_2=u_4=u_6$. Then, from (\ref{6l_1^7}) we have $y_2+y_4+y_6=0$, i.e. $y_2=y_4=y_6$. From $u_2+u_4+u_6=0$ and (\ref{6l_0^7}) we have $u_1+u_3+u_5=0$, i.e. $u_1=u_3=u_5$. Hence, there are $a,b,r\in\mathbb{F}_3^{\ast}$ such that
 \begin{align}\label{y_3,5=0}
    \epsilon=(a,r,b,-r,a,r,b,-r,a,r,b,-r).
 \end{align}

 If $y_3=0$ and $y_4y_5\neq0$, then we have $u_2+u_4+u_5+u_6=u_1+u_3=0$. Furthermore, from (\ref{sumv16eq0}) and (\ref{y_ineqy_i+1}) we see $v_1=-v_2$, $v_3=v_4=-v_5=-v_6$, $y_2=v_1$ and $y_4 =-y_5=y_6=v_3$. Then, from (\ref{6l_1^7}) we have $v_1u_2 +v_3(u_4-u_5+u_6)=0$ and thus from $u_2+u_4+u_5+u_6=0$ we see $(v_1+v_3)u_5=(v_3-v_1)(u_4+u_6)$, which implies $v_1+v_3=u_4+u_6=0$. Hence, there are $a,b,c,r\in\mathbb{F}_3^{\ast}$ such that
 \begin{align}\label{y_3=0}
    \epsilon=(a,r,b,-r,-a,-r,c,-r,-b,r,-c,r).
 \end{align}

 If $y_5=0$ and $y_3y_4\neq0$, then we have $u_2+u_3+u_4+u_6=u_1+u_5=0$. Furthermore, from (\ref{sumv16eq0}) and (\ref{y_ineqy_i+1}) we see $v_1=v_2=-v_3=-v_4$, $v_5=-v_6$, $y_2=-y_3=y_4=v_1$ and $y_6=v_5$. Then, from (\ref{6l_1^7}) we have $v_1(u_2-u_3+u_4)+v_5u_6=0$ and thus from $u_2+u_3+u_4+u_6=0$ we see $(v_1+v_5)u_3=(v_1-v_5)(u_2+u_4)$, which implies $v_1+v_5=u_2+u_4=0$. Hence, there are $a,b,c,r\in\mathbb{F}_3^{\ast}$ such that
 \begin{align}\label{y_5=0}
    \epsilon=(a,r,b,r,c,-r,-b,-r,-a,-r,-c,r).
 \end{align}

 Assume now $y_4=0$ and $y_3y_5\neq0$. Clearly, $u_2+u_3+u_5+u_6=0$ and $u_1+u_4=0$. From (\ref{sumv16eq0}) and (\ref{y_ineqy_i+1}) we have $v_1=v_2=v_3$, $v_4=v_5=v_6$, $y_2=-y_3=v_1$ and $y_5=-y_6=v_4$.

 If $v_1\neq v_4$, from (\ref{6l_1^7}) we have $u_2-u_3-u_5+u_6=0$ and thus from $u_2+u_3+u_5+u_6=0$ we see $u_2+u_6=u_3+u_5=0$. Hence, there are $a,b,c,r\in\mathbb{F}_3^{\ast}$ such that
 \begin{align}\label{y_4=01}
    \epsilon=(a,r,b,r,c,r,-a,-r,-c,-r,-b,-r).
 \end{align}

 If $v_1=v_4$, from (\ref{6l_1^7}) we have $u_2-u_3+u_5-u_6=0$ and thus from $u_2+u_3+u_5+u_6=0$ we see $u_2+u_5=u_3+u_6=0$. Hence, there are $a,b,c,r\in\mathbb{F}_3^{\ast}$ such that
 \begin{align}\label{y_4=02}
    \epsilon=(a,r,b,r,c,r,-a,r,-b,r,-c,r).
 \end{align}

 By now, we have determined all cycles of length 12 in $\Lambda_{3,3}$.

 \begin{lemma}\label{lamda33length10}
    For any tuple $\epsilon=(u_1,v_1,\ldots,u_6,v_6)$ over $\mathbb{F}_3^{\ast}$, in $\Lambda_{3,3}$ there is a cycle of type $\epsilon$ if and only if $\epsilon$ is of form among (\ref{y_3,5=0}--\ref{y_4=02}).
 \end{lemma}

 Now we consider to determine all the cycles of length 12 in $\Lambda_{4,3}$. According to Theorem \ref{path} and Lemma \ref{lamda33length10} we see that, in $\Lambda_{4,3}$ there is a cycle of type $\epsilon=(u_1,v_1,\ldots,u_6,v_6)$ if and only if $\epsilon$ is of form among (\ref{y_3,5=0}--\ref{y_4=02}) and satisfies $\Delta_4(\epsilon)=0$, where $\Delta_4(\epsilon)=x_2^2v_1+x_3^2v_2+x_4^2v_3+x_5^2v_4+x_6^2v_5$.

 If $\epsilon$ is of form (\ref{y_3,5=0}), we have
 $$\Delta_4(\epsilon)=r(a^2-(a+b)^2+(2a+b)^2-(2a+2b)^2+(2b)^2)=0.$$

 If $\epsilon$ is of form (\ref{y_3=0}), from
 $$\Delta_4(\epsilon)=r(a^2-(a+b)^2-b^2-(b+c)^2+c^2)=rb(a+c)$$
 we see $\Delta_4(\epsilon)=0$ is valid if and only if $a+c=0$. Hence, we have
 \begin{align}\label{y_3=0k4}
    \epsilon=(a,r,b,-r,-a,-r,-a,-r,-b,r,a,r).
 \end{align}

 If $\epsilon$ is of form (\ref{y_5=0}), from
 $$\Delta_4(\epsilon)=r(a^2+(a+b)^2-(a+b+c)^2-(a+c)^2-c^2)=rc(b-a)$$
 we see $\Delta_4(\epsilon)=0$ is valid if and only if $a=b$. Hence, we have
 \begin{align}\label{y_5=0k4}
    \epsilon=(a,r,a,r,c,-r,-a,-r,-a,-r,-c,r).
 \end{align}

 If $\epsilon$ is of form (\ref{y_4=01}), from
 $$\Delta_4(\epsilon)=r(a^2+(a+b)^2+(a+b+c)^2-(b+c)^2-b^2)=ra(b-c)$$
 we see $\Delta_4(\epsilon)=0$ is valid if and only if $b=c$. Hence, we have
 \begin{align}\label{y_4=01k4}
    \epsilon=(a,r,b,r,b,r,-a,-r,-b,-r,-b,-r).
 \end{align}

 If $\epsilon$ is of form (\ref{y_4=02}), from
 $$\Delta_4(\epsilon)=r(a^2+(a+b)^2+(a+b+c)^2+(b+c)^2+c^2)=r(ab+bc-ac)$$
 we see $\Delta_4(\epsilon)=0$ is valid if and only if $ab+bc-ac=0$, which is equivalent to $ab=bc=-ac$, i.e. $a=c=-b$. Hence, we have
 \begin{align}\label{y_4=02k4}
    \epsilon=(a,r,-a,r,a,r,-a,r,a,r,-a,r).
 \end{align}

 Now, we can determine all the girth cycles of $\Lambda_{k,3}$ for $4\leq k\leq 8$.

 \begin{theorem}\label{lambda-k3}
   Let $\epsilon=(u_1,v_1,\ldots,u_6,v_6)$ be a tuple over $\mathbb{F}_3^{\ast}$.
   \begin{enumerate}
      \item $\Lambda_{4,3}$ has a cycle of type $\epsilon$ if and only if $\epsilon$ is of form (\ref{y_3,5=0}) or among (\ref{y_3=0k4}--\ref{y_4=02k4}).
      \item For $k=5,6$, $\Lambda_{k,3}$ has a cycle of type $\epsilon$ if and only if there are $a,b,c,d \in \mathbb{F}_3^{\ast}$ with $(a+b)(c+d)=0$ such that
      \begin{align}\label{cycle53}
      	\epsilon=(a,c,b,d,a,c,b,d,a,c,b,d).
      \end{align}
      \item $\Lambda_{7,3}$ has a cycle of type $\epsilon$ if and only if there are $a,b,c\in \mathbb{F}_3^{\ast}$ such that
      \begin{align}\label{cycle73}
      	\epsilon=(a,c,b,-c,a,c,b,-c,a,c,b,-c).
      \end{align}
      \item $\Lambda_{8,3}$ has a cycle of type $\epsilon$ if and only if there are $a,c\in \mathbb{F}_3^{\ast}$ such that
      \begin{align}\label{cycle83}
      	\epsilon=(a,c,-a,-c,a,c,-a,-c,a,c,-a,-c).
      \end{align}
      \item For $4\leq k\leq 8$, $g(\Lambda_{k,3})=12$.
      \item $g(\Lambda_{9,3})\geq14$.
   \end{enumerate}
 \end{theorem}

 \begin{proof}
 	The first result follows from the argument preceding this theorem.
 	
 	To show the second result, we assume $\Lambda_{4,3}$ has a cycle of type $\epsilon$. By directly computing, one can get
 	\begin{align*}
 		l_5^{(7)}=-\sum_{2\leq r\leq s\leq5}y_ru_rv_sx_{s+1}=\begin{cases}
 			0,   & \text{if } \epsilon \text{ is of form (\ref{y_3,5=0}) or (\ref{y_4=02k4}),}\\
 			r^2ab,   & \text{if } \epsilon \text{ is of form (\ref{y_3=0k4}),}\\
 			r^2ac,   & \text{if } \epsilon \text{ is of form (\ref{y_5=0k4}),}\\
 			-r^2ab,    & \text{if } \epsilon \text{ is of form (\ref{y_4=01k4}).}\\
 		\end{cases}
 	\end{align*}
    Hence, $\Lambda_{5,3}$ has a cycle of type $\epsilon$ if and only if $\epsilon$ is of form (\ref{y_3,5=0}) or (\ref{y_4=02k4}), i.e. (\ref{cycle53}) is valid for some $a,b,c,d \in \mathbb{F}_3^{\ast}$ with $(a+b)(c+d)=0$. Furthermore, for such tuple $\epsilon$ we have
    $$l_6^{(7)}=-\sum_{1\leq s_1\leq t_1<s_2\leq t_2\leq6}u_{s_1}v_{t_1}u_{s_2}v_{t_2}=0,$$
    which means $\Lambda_{6,3}$ also has a cycle of type $\epsilon$.

    To show the third result, we assume $\Lambda_{6,3}$ has a cycle of type $\epsilon$. From the second result, we see $\epsilon$ is of form (\ref{cycle53}) for some $a,b,c,d \in \mathbb{F}_3^{\ast}$ with $(a+b)(c+d)=0$, then we have
    $$l_7^{(7)}=-\sum_{1\leq s_1<t_1\leq s<t_2\leq s_2\leq 6}v_{s_1}u_{t_1}v_su_{t_2}v_{s_2}=-ab(c+d).$$
    Hence, $\Lambda_{7,3}$ has a cycle of type $\epsilon$ if and only if $c+d=0$, i.e. (\ref{cycle73}) is valid for some $a,b,c\in \mathbb{F}_3^{\ast}$.

    To show the fourth result, we assume $\Lambda_{7,3}$ has a cycle of type $\epsilon$. From the third result, we see $\epsilon$ is of form (\ref{cycle73}) for some $a,b,c\in \mathbb{F}_3^{\ast}$, then
    we have
    $$l_8^{(7)}=\sum_{1\leq t_1\leq s_1<t\leq s_2<t_2\leq 6}u_{t_1}v_{s_1}u_tv_{s_2}u_{t_2}=-(a+b).$$
    Hence, $\Lambda_{8,3}$ has a cycle of type $\epsilon$ if and only if $a+b=0$, i.e. (\ref{cycle83}) is valid for some $a,c\in \mathbb{F}_3^{\ast}$.

    The fifth result is a direct corollary of the first four results.

    The last result is a corollary of the known bound given in \cite{Lazebnik95}. However, here we give a simple proof for it by using the fourth result.
    We assume $\Lambda_{8,3}$ has a cycle of type $\epsilon$, then $\epsilon$ is of form (\ref{cycle83}) and then from
    $$l_9^{(7)}=\sum_{1\leq s_1<t_1\leq s<t\leq s_2< t_2\leq 6}v_{s_1}u_{t_1}v_{s}u_{t}v_{s_2}u_{t_2}=-ac\neq 0$$
    we see there is no cycle of length 12 in $\Lambda_{9,3}$. Hence, we have $g(\Lambda_{9,3})\geq14$.
 \end{proof}

We note that $g(\Lambda_{5,3})=12$ has been pointed out in \cite{Fredi95} without proof.

\section{Concluding Remarks}

Note that one can also rewirte (\ref{4j+2}) and (\ref{4j+3}) without use of $v_i$, which is an information about the vertex $\langle r^{(i+1)}\rangle$ next to the present vertex $[l^{(i+1)}]$, as the following
\begin{align*}
          l_{4j+2}^{(i+1)}&=y_{i}l_{4j}^{(i+1)}-\rho_{i-j-2}(u_1,v_1,\ldots,u_{i-1},v_{i-1}),\\
          l_{4j+3}^{(i+1)}&=y_{i}l_{4j+1}^{(i+1)}-\rho_{i-j-3}(v_1,u_2\ldots,v_{i-2},u_{i-1},v_{i-1}),
\end{align*}
respectively.

Since $\Lambda_{k,q}$ is edge-transitive, the girth cycles in $\Lambda_{k,q}$ containing any given edge are determined indeed by Theorems~\ref{lamda34}--\ref{lambda-k3} for a few small $k$'s.
However, it is still not transparent how to count all the girth cycles in these graphs.

For example, let $\Phi$ denote the set of backtrackless walks $[l_1]\langle r_1\rangle[l_2]\langle r_2\rangle[l_3]$ of length 5 in $\Lambda_{3,3}$.
Since $g(\Lambda_{3,3})=8$ and in $\Lambda_{3,3}$ there is a 8-cycle of type $\epsilon$ if and only if (\ref{3v_12eq0}) is valid for some $r,s,t\in\mathbb{F}_3^*$, according to the symmetry of $\Lambda_{3,3}$ we see that any walk
$\Gamma_5=[l_1]\langle r_1\rangle[l_2]\langle r_2\rangle[l_3]\in \Phi$
is a path, namely consisting of distinct vertices, and there is a unique walk
$\Gamma_5'=[l_1]\langle r_1'\rangle[l_2']\langle r_2'\rangle[l_3]\in\Phi$
such that
$$[l_1]\langle r_1\rangle[l_2]\langle r_2\rangle[l_3]\langle r_2'\rangle[l_2']\langle r_1'\rangle$$
is a girth cycle of $\Lambda_{3,3}$. Moreover, such girth cycle is uniquely determined by the walk $\Gamma_5$.
Since the total number of walks in the set $\Phi$ is $3^3\times 3\times 2\times 2\times 2=648$
and each girth cycle of $\Lambda_{3,3}$ contains exact 8 walks in $\Phi$, we see there are $648/8=81$ girth cycles in $\Lambda_{3,3}$.

However, this method is not effective even for $\Lambda_{3,q}$ with $q>3$ since different girth cycles
in such graphs may have a common backtrackless walk of length 5.

\end{document}